\documentclass{article} \usepackage{amsmath} \usepackage{amssymb}
\newcommand{\di}{\displaystyle} \newcommand{\al}{\alpha}
 \newcommand{\be}{\beta}

\usepackage{graphicx}
\newcommand{\ti}{\times}

\usepackage{graphicx}
\newcommand{\F}{\mathbb{F}}

\usepackage{fullpage,url,amsmath,amsfonts,amssymb,tedmath}
\title{Algebraic constructions of  LDPC codes with no short cycles.}
\date{} 
\author{Ted Hurley\thanks{National University of Ireland, Galway, Ireland.
ted.hurley@nuigalway.ie }\and 
{Paul McEvoy \thanks{TechnologyFromIdeas, Old Kilmeaden Road,
Waterford,
Ireland. paul.mcevoy@technologyfromideas.com}
}\and 
{Jakub Wenus \thanks{TechnologyFromIdeas, Old Kilmeaden Road,
Waterford,
Ireland. jakub.wenus@technologyfromideas.com}
}
}
\begin{document}
\maketitle

\begin{abstract} An algebraic group ring method for constructing codes 
with no   short cycles in the check matrix  is derived.
It is shown that the matrix of a group ring element has no short
cycles if and only if the collection of group differences of this
element has no repeats.  
When applied to elements in the group ring with small support this
gives a general method for constructing and analysing 
low density parity check (LDPC) 
 codes with no short cycles from group rings. 
Examples of LDPC codes with no short cycles 
are constructed from group ring elements  
and these are simulated and compared
with known LDPC codes, including those adopted for wireless standards.

\end{abstract}  

\section{Introduction}

An LDPC code is a code where the check matrix has only a small number
of non-zero elements in each row and column. These were introduced by 
Gallager in \cite{gallager}, 
 expanded further by Tanner in \cite{tanner},
and rediscovered \cite{neal1,neal}  
 and expanded on by MacKay and Neal;
details can now be found in \cite{mackay}. Structured LDPC codes
usually use various  types of
combinatorial objects such as designs or algebraic geometry.  
  LDPC codes have often been produced
by randomised techniques, but there has
been recent activity in the area of algebraic
constructions~\cite{tang04-algebraicldpc, milenkovic04blockcirculantldpc,
tanner2004-qcldpc}. 
Having no
short cycles in the (Tanner) graph of the check matrix of an LDPC code 
has been shown to dramatically improve the performance of
the code. Short cycles in an LDPC code 
 deteriorate the performance of the decoding algorithms and having no
 short cycles may in effect increase the distances of such codes.

Here a group ring method for the construction and analysis of LDPC
codes with no short cycles is presented.

\subsection{Group ring method} 
Cyclic and related 
codes owe much of their structure and properties because they  
 occur as ideals or modules within a cyclic group ring. 
 So also general group rings may be used to  
construct, analyse and give structure to 
many other types of codes and  may be used to construct codes of a
particular type or with a particular property. 
  Group ring module codes are  obtained from either zero-divisors or
units in a group
ring $RG$ as described in \cite{hur} or \cite{hur1}. 
Thus elements $u,v \in RG$ are considered with either $uv =
0$ (zero divisors) or $uv = 1$ (units), 
 where $1$ denotes the identity of $RG$, and   codes are derived therefrom.

The unit-derived group ring code method is particularly useful and has
great 
flexibility while still retaining much  of the algebraic structure. This
method is employed here to directly and
algebraically construct low
density parity check (LDPC) codes with no short cycles in their
(Tanner) graphs. Zero-divisors
may also be used but the 
unit-derived method  has advantages and has 
less theoretical complications. 

 Using an injection $\phi: RG
 \rightarrow R_{n\ti n}$,  as for
example in \cite{hur3}, from the group ring $RG$ with $|G|=n$  into the
ring of $n\ti n$ matrices over $R$, corresponding matrix codes are
obtained from the group ring codes. In this injection the notation 
$\phi(v) = V$ is used  
so that the capital letter $V$ in the matrix ring corresponds to the
 lower case letter $v$ in the group ring. 

Thus we are lead to consider  $u,v \in RG$ with $uv= 1= vu$ from which the
codes are defined. Certain rows are chosen from the matrix  $U$ of the
unit $u$ to
form the generator matrix of a code and
the corresponding columns are deleted from the matrix $V$ of the
inverse $v$ of $u$ to form the
check matrix of this code.  

If $v$  has short support then its corresponding matrix
$V$ has only a small number of elements in each row and column and is
thus low density. `Short support' of a group ring element $v$ means that
only a small number (compared to the size of the group) of the
coefficients in $v$ are non-zero.   
If  the  check matrix  of a group ring code 
is derived  from a group ring check element  $v$ 
with small support  
then  the resulting code will be an LDPC code. 

It is  determined here where precisely the short cycles can occur
in the matrix of a  group ring element.  It is
then easy to 
construct group ring elements, and group ring elements of small
support, 
which will have no short cycles anywhere
in their 
matrices. Using  group ring elements with small 
support and with no short cycles in their matrices,  LDPC codes with no short
cycles are constructed. 

Thus LDPC codes with no short cycles can be constructed by the
following algebraic group ring method: 
\begin{enumerate} \item  Construct $u,v$ is a group ring $RG$ with
 $uv = 1$ and such that $v$ has small support compared to the size
  $|G|$ of the group $G$.  These can be
  constructed with some property in mind. Units in group rings abound
  and are easy to construct. \item  Decide on the rate $\frac{r}{|G|}$ of the
  code required. This is often
decided by reference to $u,v$ and their structures. \item Choose $r$
rows of $U$ with which to construct the generator matrix and delete
the corresponding $r$ columns of $V$ to form the check matrix. This
  gives a rate $\frac{r}{|G|}$ code.  \item
If $V$ has no short cycles at all, which can be ensured by
\thmref{thm2} below, 
then any choice of columns  of $V$ and consequent choice of rows of $U$
will give an LDPC code with no short cycles. \item If $V$ has short
cycles,  it
may be possible to avoid these by deleting appropriate columns 
and still obtain an LDPC code with no short cycles. \end{enumerate} 

It is of obviously easier to ensure there are no short cycles in the
resulting LDPC codes if the original (check) 
element $v$ from which the code is constructed has no short cycles at
all in its matrix $V$. This can be ensured in the construction of $v$ 
by \thmref{thm2} below.

Another advantage of codes derived from units is that there is a huge
choice of columns from which  to form a check 
matrix of a code.  
For example suppose from a unit  of size 
 $1000\ti 1000$ with no short cycles and low density  a $(1000,500)$
 code  is required.  There are $\binom{1000}{500}$, which is 
of the order of 
$2^{500}$, choices from the $1000\ti 1000$ matrix with which to form
 the code and each code is  low density and has no
short cycles.  From the nature of the independence of any set of rows
(or columns) in a unit (= non-singular) matrix each code derived is
different.  

One could thus 
envisage a hybrid whereby a random construction is performed within the
parameters of an algebraic construction.

\subsection{Examples and simulations} 
Examples, simulations and comparisons are given in \sref{sims}. 
These compare extremely well with
existing LDPC codes and in many cases outstrip them. An example is
also given of a unit from which $10$ random LDPC codes of rate $1/2$ are
constructed. These are then simulated and all of them performed
well. The potential applications of having random LDPC codes with no
short cycles derived
from a single unit and all performing well are obvious.  
In addition  comparisons of Bit Error Rate (BER) and Block Error Rate (BLER)
performance of LDPC codes defined in the 802.11n \& 802.16e standard
with equivalent codes generated by the present  method are given in
\sref{tfi}.

BER is not everything and often fast and
power-efficient coding is more important than performance. 
The group ring  method for LDPC codes needs only a relatively few
initial parameters and can re-create the matrix line-by-line without
the need to store the whole structure in memory. 
The method has thus 
in addition applications where low power and low storage are
requirements.  

\subsection{Notation} 

$RG$ denotes the group ring of the group $G$ over the ring $R$; when
$R$ is a field, $RG$ is often referred to as a {\em group algebra}. No deep
knowledge of group rings is required but familiarity with the ideas of
{\em units, zero-divisors} in rings is assumed. For further information
on group rings see \cite{sehgal}. $C_n$ denotes the cyclic group of
order $n$ and $H\ti K$ denotes the direct product of groups $H,K$. 

The words `graph' and `short cycles' is used but as now explained no
knowledge of graph theory is required and the problem of avoiding
short cycles reduces to looking at a property of matrices.

For any matrix $H = (h_{ij})$ the Tanner graph \cite{tanner} of $H$ is
  a bipartite graph $K = V_1 \cup V_2$ where $V_1$ has one vertex for
  each row of $H$ and $V_2$ has one vertex for column in $H$ and there
  is an edge between two vertices $i,j$ exactly when $h_{ij} \not
  = 0$. A short cycle in the (Tanner) graph of a matrix is a cycle of
  length $4$. 

Thus a matrix has no short cycles in its graph if and only 
the intersection of positions
  in which two columns have non-zero values is at most
  $1$. This definition is used when considering the absence or
  otherwise of short cycles and thus no deep graph theory is involved.   

\section{Avoiding short cycles} 

Avoiding short cycles in the (Tanner) graph of the check matrix of a code is
important, particularly for LDPC (low density parity check) codes.
 
Specifically here, 
   necessary and sufficient conditions are given on the  group ring
 element $v$ in terms of  
 the group elements  with non-zero coefficients occurring in it 
  so that its corresponding matrix $V$ has no short cycles. 
A mathematical proof is  given.

Some special cases,  such as when $G$ is cyclic
or abelian, of the general result are easier to describe and useful in 
practice and these are used as examples and illustrations of the general
results.

\subsection{Collections of differences, special case}
Collections of differences are usually defined with respect to a set a
non-negative integers, see for example \cite{lint}. Collections of
group differences are  defined in \sref{general} and the collections
of (integer) differences are special cases of these when the group is a
cyclic group.

The integer definition is recapped here and used to give 
examples of the general definition.
  
Let $S = \{i_1,i_2, \ldots, i_r\}$ be a set of non-negative unequal integers
and $n$ an integer with $n > i_j$ for all $j=1, 2, \ldots, r$.

Then the {\em collection of cyclic differences of $S \mod n$} is
defined by $DS(n) = \{
i_j - i_k \mod n | 1\leq j,k \leq r, j\not = k\}$. This collection has 
possibly repeated elements.

For example if $S=\{1,3,7,8\}$ and $n = 12$ then $DS(12) =
\left\{\begin{array}{rrr} 2 & 6 & 7\\ &4&5 \\ &&1 \\ 10 & 6 & 5 \\ &8&7 \\
&&11 \end{array}\right\} =  
\{2,6,7,4,5,1,10,6,5,8,7,11\}$. In this case $6,7,5$  occur twice.

If $|S| = r$ then counting repeats $|DS(n)| = r(r-1)$.

\subsection{Cyclic group ring differences}\label{sec:cyclic1}
Consider the group ring $RC_n$ where $C_n$ is the cyclic group of
order $n$ generated by $g$. 
Suppose  $u=
\al_{i_1}g^{i_1} + \al_{i_2}g^{i_2} + \ldots + \al_{i_r}g^{i_r} \in RC_n$ with
$\al_{i_j} \not = 0$ (and $0\leq i_j< n$).

For each $g^i,g^j$ in $u$ with non-zero coefficients form $g^ig^{-j},
g^jg^{-i}$ and define $DS(u)$ to be the collection of all such $g^ig^{-j},
g^jg^{-i}$. 

Set $S= \{i_1, i_2 , \ldots, i_r\}$ and define the collection of
cyclic differences $DS(n)$ as above. It is clear that $DS(n)$ and $DS(u)$ are 
equivalent, the
only difference being in the notation used. 
The proof of the following theorem is a direct corollary of the more
general \thmref{thm1} below.

\begin{theorem}\label{thm:thm2} $U$ has no 4-cycles in its graph if and only if
  $DS(u)$ has no repeated elements.
\end{theorem}

\subsubsection{Example}
Set $u= 1 + g + g^3+ g^7$ in $\Z_2C_{15}$. The collection of
differences is formed from $\{0,1,3,7\}$ and is thus  

$DS(u) =
\{1,3,7,2,6,4,14,12,8,13,9,11\}$ 
which has no repeats. Hence  the matrix
formed from $u$, which is circulant in this case,  has no short cycles.

Set $u = 1 +g+ g^3+g^7$ in $\Z_2C_{13}$. 

The collection of differences formed from $\{0,1,3,7\}$ is
$\{1,3,7,2,6,4, 12,10,6,11,7, 9\}$ and has repeats $6,7$. 

Thus the matrix formed
  from $u$ has short cycles -- but we can identify where they occur.

\subsection{Collection of differences in a general group
  ring}\label{sec:general} 
Let $RG$ denote the group ring of the group $G$ over the ring $R$.
Let $G$ be  listed by $G = \{g_1, g_2, \ldots, g_n\}$.

Let $u = \di\sum_{i=1}^{n} \al_ig_i$ in $RG$.

For each (distinct) pair $g_i,g_j$ occurring in $u$ with non-zero 
coefficients, form the (group) {\em differences} $g_ig_j^{-1},
g_jg_i^{-1}$. Then {\em the collection of difference} of $u$, $DS(u)$, 
consists of all such differences. Thus: 

$DS(u) = \left\{ g_ig_j^{-1}, g_jg_i^{-1} | g_i \in G, g_j\in G, 
i \not = j, \al_i \not =
0,  \al_j\not =
0\right\}$.

Note that the collection of differences  of $u$ consists of group
elements and for
 each $g,h$, $g\not = h$, occurring with non-zero coefficients in $u$ 
both $gh^{-1}$ and its inverse $hg^{-1}$ are formed
 as part of the collection of differences.

\begin{theorem}\label{thm:thm1}
   
The matrix $U$ has no short cycles in its graph if and only if 
$DS(u)$ has no repeated (group) elements.
\end{theorem}
\begin{proof}
 The rows of $U$ correspond in order to $ug_i, i= 1,
  \ldots, n$, see \cite{hur}.

Then $U$ has a 4-cycle 

$\Longleftrightarrow$ 

for some $i\not = j$ 
and some $k \not = l$, the  coefficients of $g_m, 
g_l$, in $ug_i$ and $ug_j$ are nonzero.

$\Longleftrightarrow$

$ug_i = ... + \al g_k + \be g_l + \ldots $

and 

$ug_j = ... + \al_1g_k + \be_1 g_l + \ldots$ 

$\Longleftrightarrow$ 
 
$u = ...+\al g_{k}g_i^{-1} + \be g_{l}g_i^{-1} + \ldots $

and 

$u = ...+\al_1g_{k}g_j^{-1} + \be_1 g_{l}g_j^{-1} + \ldots$.

$\Longleftrightarrow$

 $DS(u)$ contains both $g_kg_i^{-1}g_ig_l^{-1} = g_kg_l^{-1}$ 
 and $g_kg_j^{-1}g_l^{-1} = g_kg_l^{-1}$.
 
This happens if and only if $DS(u)$ has a repeated element.
\end{proof}

\subsection{Repeated elements}{\label{sec:rep}}
Suppose now $u$ is such that $DS(u)$ has repeated elements.

Hence $u = ... + \al_mg_m + \al_rg_r + \al_pg_p + \al_qg_q + ...$, 
where the displayed $\al_i$ are not zero, so
that $g_mg_r^{-1} =
g_pg_q^{-1}$. The elements causing a short
cycle are displayed and  note that the elements $g_m, g_r, g_p, g_q$
are not necessarily in the order of the listing of $G$.  

Since we are interested in the graph of the
element and thus in the non-zero coefficients, replace a  non-zero
coefficient by the coefficient 1. 
 Thus write $u = ... + g_m + g_r + g_p + g_q + ...$ so that $g_mg_r^{-1} =
g_pg_q^{-1}$.  

Include the case where one  $p, q$ could be one of $m,r$ in which
case it should not be listed in the expression for $u$.

Then  $ug_m^{-1}g_{p} = ..+ g_p + g_{r}g_m^{-1}g_p.+... =
  ...+ g^p + g^q + ..$

 and  $ug_p^{-1}g_m = ....+
  g_m +  g_{q}g_p^{-1}g_m = ... + g_m + g_r+ ...$.

(Note that  $ug_m^{-1}g_{p} =
  ug_r^{-1}g_{q} $ and  $ug_p^{-1}g_m = ug_q^{-1}g_r$)

Thus to avoid short cycles, do not use the  row determined by
$g_m^{-1}g_p$  or the row determined by $g_p^{-1}g_m$ in $U$ 
if using the first row or
in general if $g_i$ row
occurs then $g_ig_m^{-1}g_p$, and $g_ig_p^{-1}g_m$  rows must not occur.
 
Similarly when $DS(u)$ has repeated elements by avoiding certain
columns in $U$, it is possible to finish up with a matrix without
short cycles.

\subsection{Special group cases}


The special case when $G=C_n$ was dealt with in \sref{cyclic1}.

Let $G=C_n\times C_m$ be the direct product of cyclic groups $C_n,
C_m$  generated by $g,h$ respectively. These groups are particularly
useful in practice. 

List the elements of $G$ by $\{1, g, g^2, \ldots, g^{n-1}, h, hg,
hg^2, \ldots, hg^{n-1}, \ldots, h^{m-1}, h^{m-1}g, \ldots,
h^{m-1}g^{n-1}\}$.
  
Then every element in  $RG$ is of the form $u= a_0  + ha_1+ \ldots +
h^{m-1}a_{m-1}$  with each $a_i \in C_n$. 
 The collection  of
differences of $u$ is easy to determine and elements with no repeats
in their  collection of
differences are thus easy to construct.

Relative to this listing the matrix of an element in $RG$ is a
circulant-by-circulant matrix of size $mn\times mn$, \cite{hur3}.

Another particularly useful group which is relatively easy to work
 with is the 
 dihedral group $D_{2n}$ given by 
$\langle a, b | a^2=1, b^n = 1, ab = b^{-1}a \rangle $, \cite{sehgal}. This
 group is non-abelian for $n\geq 3$. 
Every element $u$ in $RD_{2n}$ may be written as $u= f(b) + ag(b)$
 with $f(b),g(b) \in RC_n$ where $C_n$ is generated by $b$. The collection of
 differences of $u$ is easy to determine. The
 corresponding matrix $U$ of $u$ has the form $\left(\begin{array}{rr} A & B
 \\ B & A\end{array}\right)$ where $A$ is circulant and $B$ is reverse
 circulant, \cite{hur3}. This 
 gives non-commutative matrices and non-commutative codes.

\section{Examples, simulations and comparisons}\label{sec:sims}

 In this section examples and simulations of the method are given 
and some comparisons are made with known  codes.
The sizes of the examples are chosen  in order to compare with known 
 examples. However there is no theoretical limit on size, the
 constructions are easy to perform and there is complete freedom as to
 choice of rows or columns to delete from a particular unit in order 
obtain LDPC codes with no short cycles. 

The simulations compare very favourably
 with known  examples and in some cases outstrip these.

This  algebraic method for
 construction  has other advantages such as for applications
 where low storage and low power are requirements. 
  The code may be
 stored by an algebraic formula with few parameteres 
and the check matrix restored as
 required line-by-line without the need to store the whole structure
 in memory.  

\subsection{The examples generally} 
In general  the examples  are taken from unit-derived codes within 
$\Z_2(C_n\ti C_4)$, where $\Z_2 = \F_2$ is  the field of two elements. 

The matrices derived are then submatrices of circulant-by-circulant
matrices and are easy to program.

Assume that $C_n$ is generated by $g$ and  $C_4$ is
generated by $h$. 

Every element in the group ring is then of the form:
$\di\sum_{i=0}^{n-1}(\al_ig^i + h\be_ig^i + h^2\gamma_ig^i +
h^3\delta_ig^i)$,
with $\al_i,\be_i,\gamma_i,\delta_i \in \Z_2$.

\subsection{(96,48) examples}\label{sec:96}
 These are derived from   $\Z_2(C_{24}\ti C_4)$. 

The check element
$v = g^{24-9} + g^{24 - 15} +g^{24-19} + hg^{24-3} + hg^{24-20}+
   h^2g^{24-22} + h^3g^{24-22} + h^3g^{24-12}$ is used to define the
     LDPC code TFI-96-59-8. 

It is easy to check from \thmref{thm1} that  $v$ has  no short cycles in
its matrix $V$.

A pattern to delete half the columns from the matrix $V$ of $v$ 
 is  chosen  to produce the rate $1/2$ code TFI-96-59-8 . 

TFI-96-59-8 is compared to pseudo-random code  
MK-96-33-964 (size=96, rate=1/2) of MacKay \cite{mak}.

\includegraphics[scale=.5]{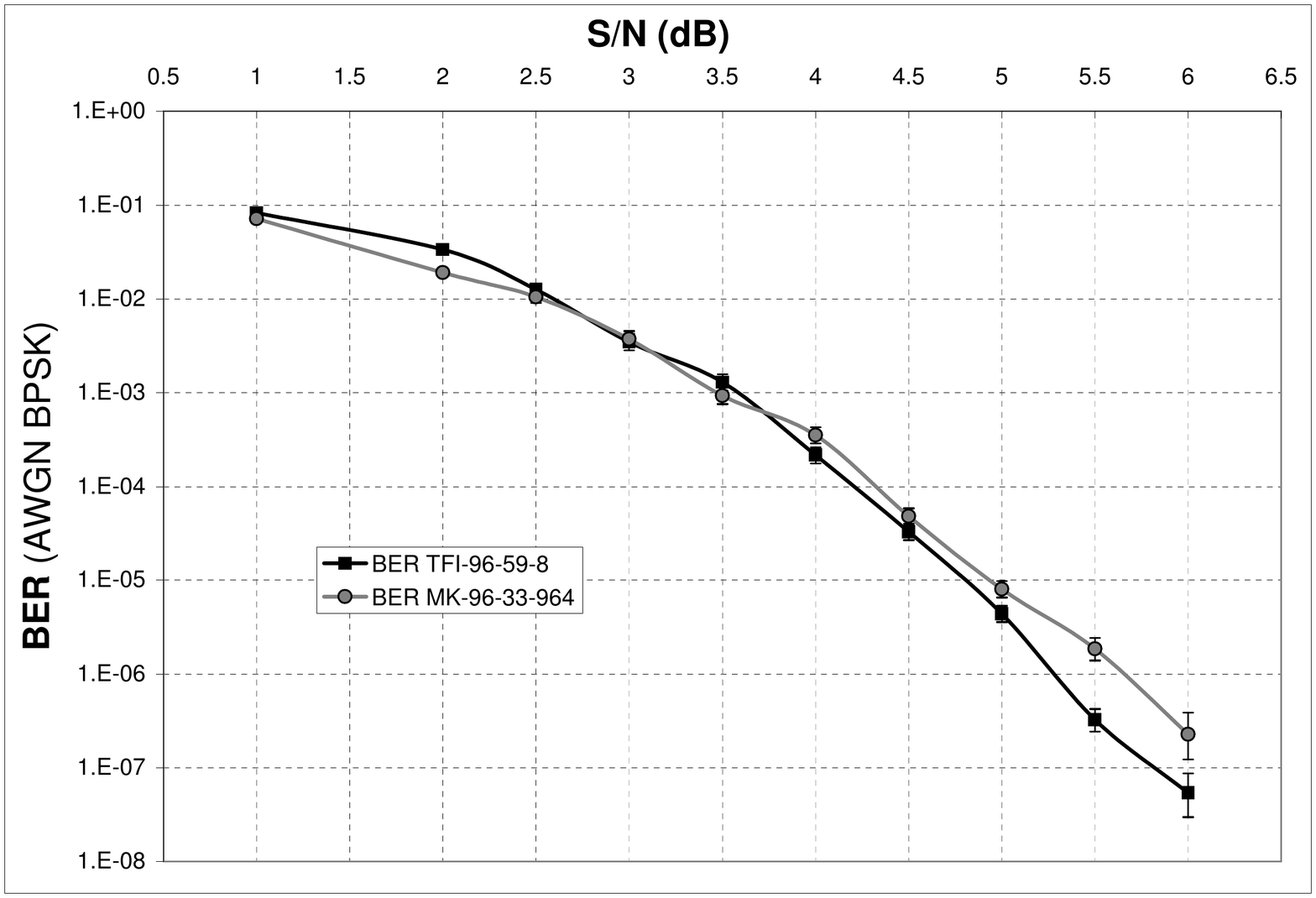}

\subsection{Random selection}
 
\includegraphics[scale=.7]{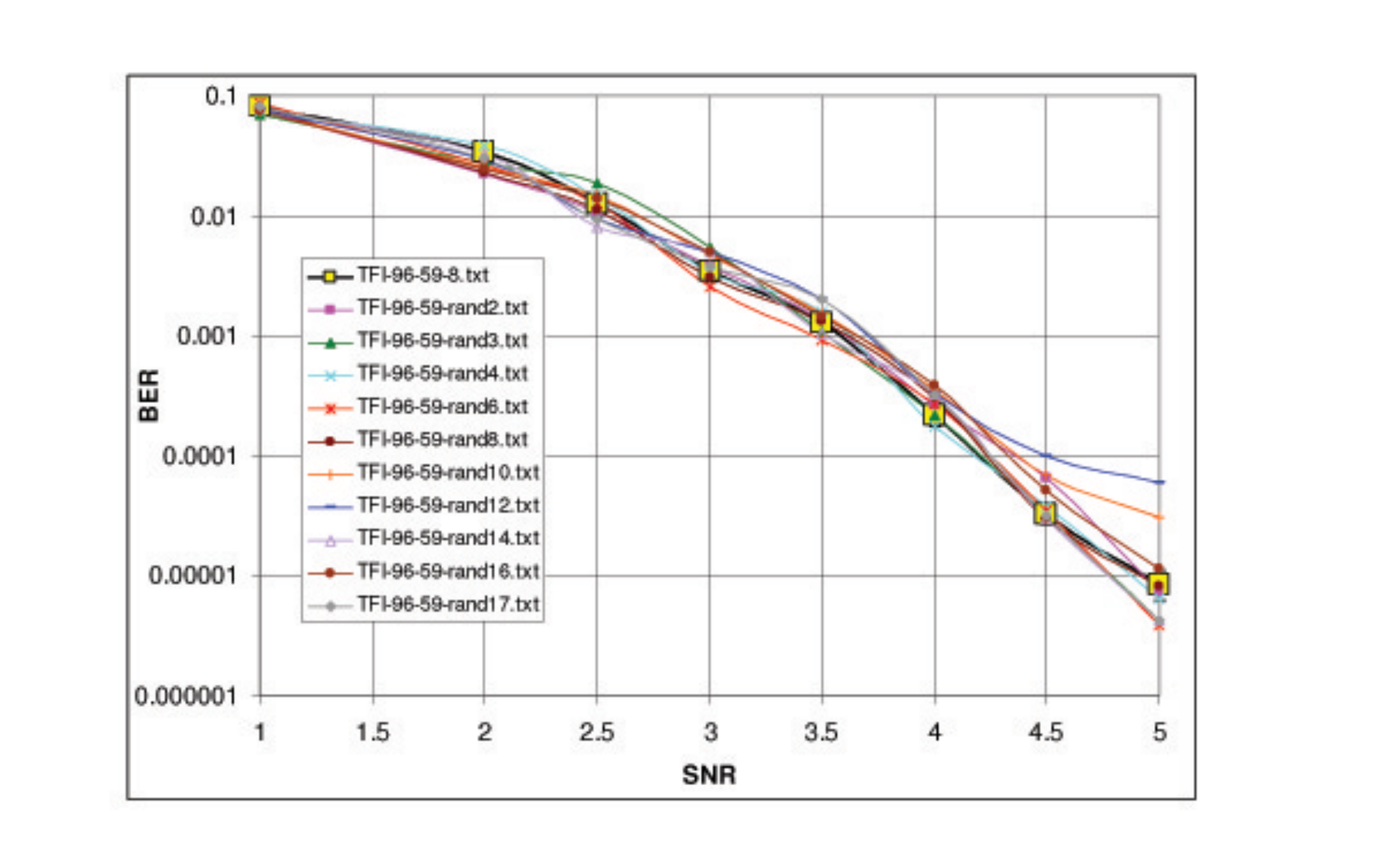}

For the above, $10$ random LDPC $(96, 49)$ codes were 
taken  from the unit $v$ in \sref{96} and simulated.
The simulation of TFI-96-59-8, used in the previous graph where it is
compared to MK-96-33-964, is included for comparison.

\subsection{(504,252) example}
The next example is derived  from $\Z_2(C_{126}\ti C_4)$.

$v = g^{126-10} + g^{126-99} + hg^{126-47}
 +h^2(g^{126-15} + g^{126-25} + g^{126-81}) + h^3(g^{126-6}+g^{126-23}
 + g^{126-64})$.

Specific column deletions are chosen from $V$ 
 to give the    
LDPC
rate $1/2$ code TFI-504-91-0. 
The performance of TFI-504-91-18 is compared to that of 
PEGReg252x504 Progressive Edge Growth, Xiao-Yu Hu,
IBM Zurich Research Labs. 

\includegraphics[scale=.326]{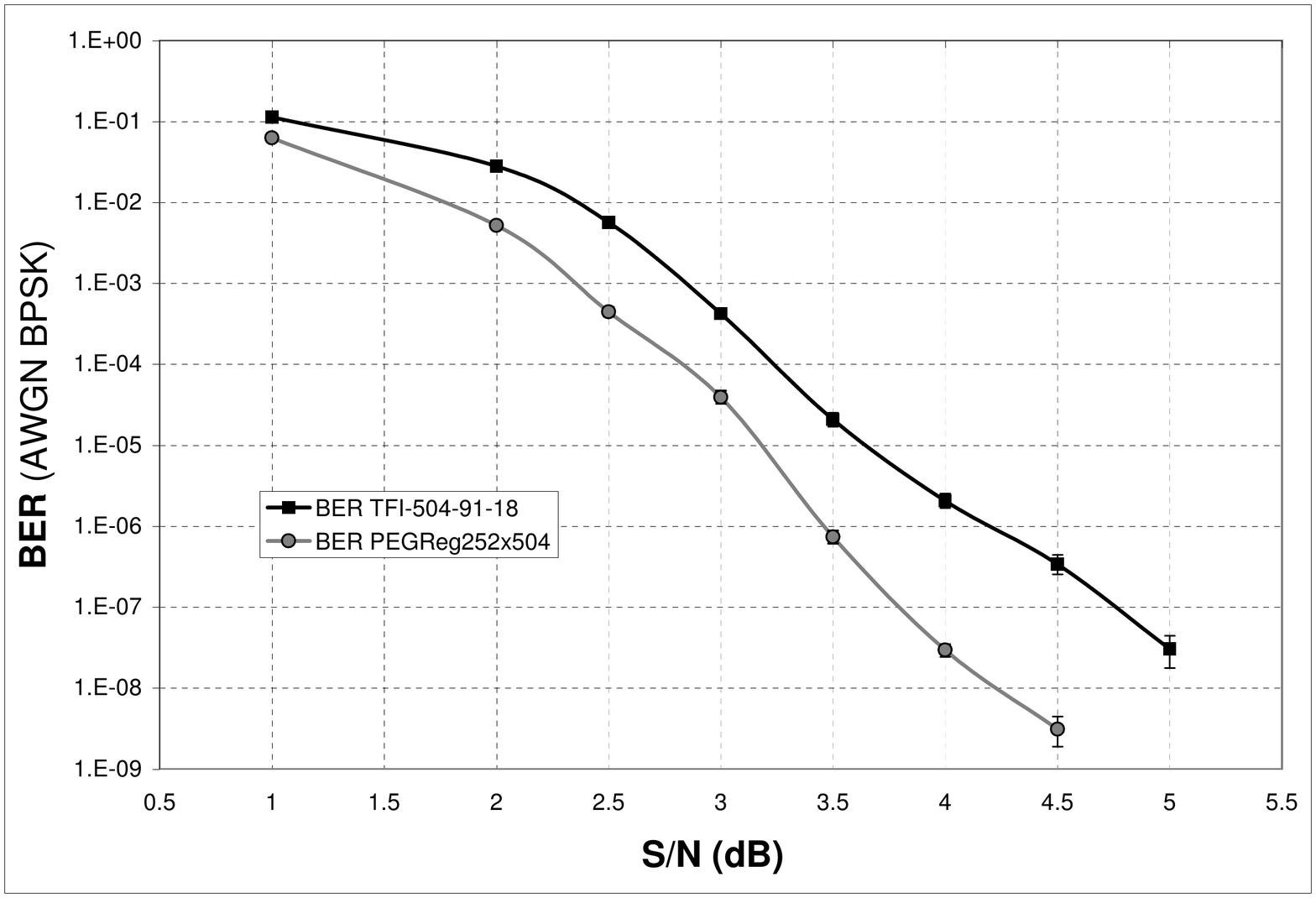}

\subsection{Size: 816; rate =1/2 and 3/4}\label{sec:ref}

Here $\Z_2(C_{204}\ti C_4)$ is used. Set  

 $ v=g^{204-75}+ h(g^{204-13} + g^{204-111} + g^{204-168})
+h^2(g^{204-29} + g^{204-34}+g^{204-170}) + h^3(g^{204-27} +
g^{204-180})$.

Half the columns of $V$ are deleted in a specific manner to get 
 the TFI-816-0p5-29-4 rate 1/2 code.

The same  $v$ is taken and specific three
quarters of the columns of $V$ 
are deleted to get the $3/4$ rate $(816,612)$ LDPC
code TFI-816-0p75-29-4.

In the first graph the performances of 
 TFI-816-0p5-29-4 and TFI-816-0p75-29-4 are
compared.

In the second graph the performances of TFI-816-0p5-29-4  and 
 MK-816-55-156, a pseudo-random rate $1/2$ code due to MacKay,
 \cite{mak}, are compared.

\includegraphics[scale=.5]{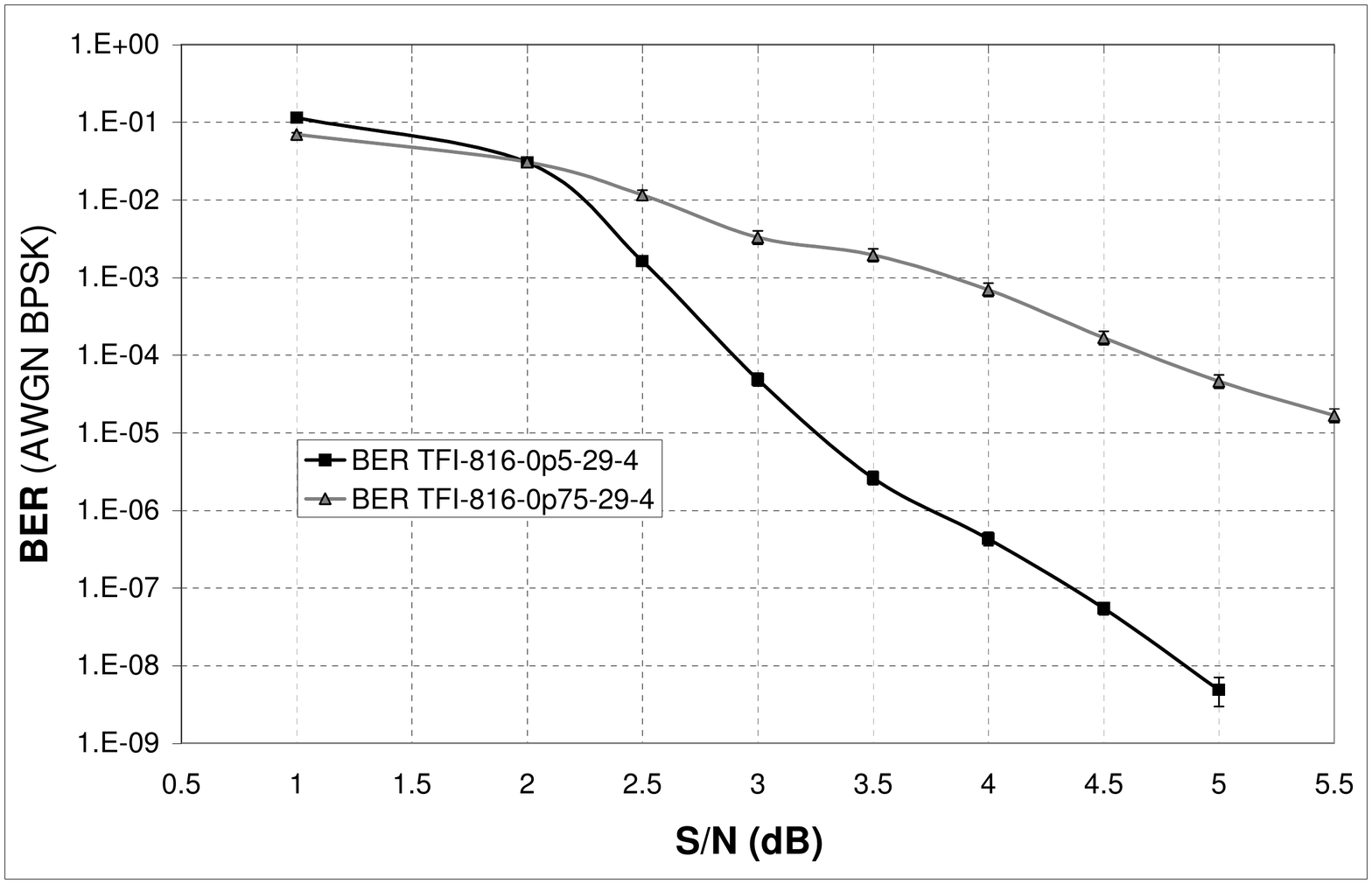}

\includegraphics[scale=.5]{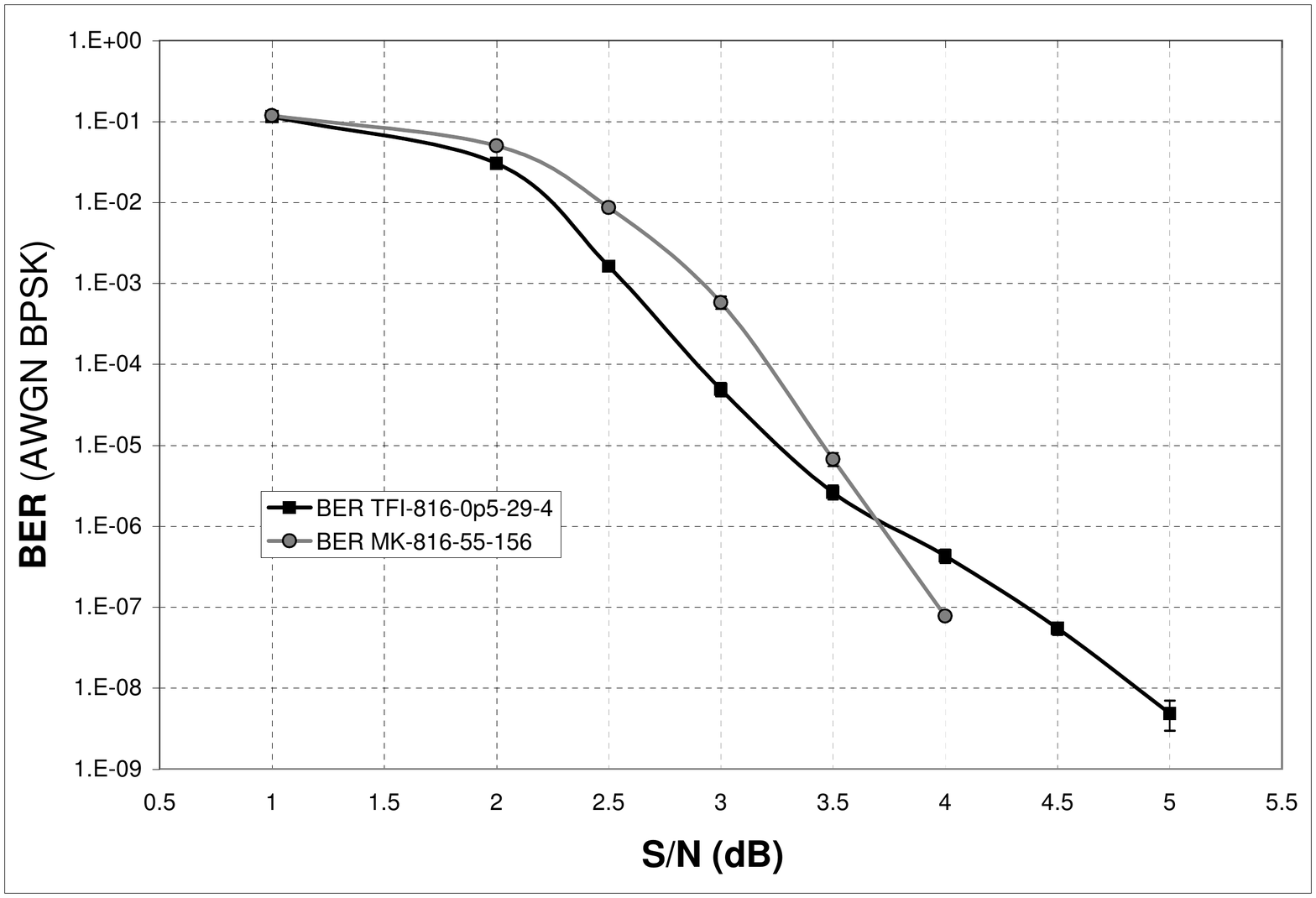}

\subsection{Industry Standards}\label{sec:tfi}

Here  comparisons of Bit Error Rate (BER) and Block Error Rate (BLER)
performance of LDPC codes defined in the 802.11n \& 802.16e standard 
with equivalent codes generated by present method are given. 

\subsubsection{Case 1, 802.11n:}
\includegraphics[scale=.64]{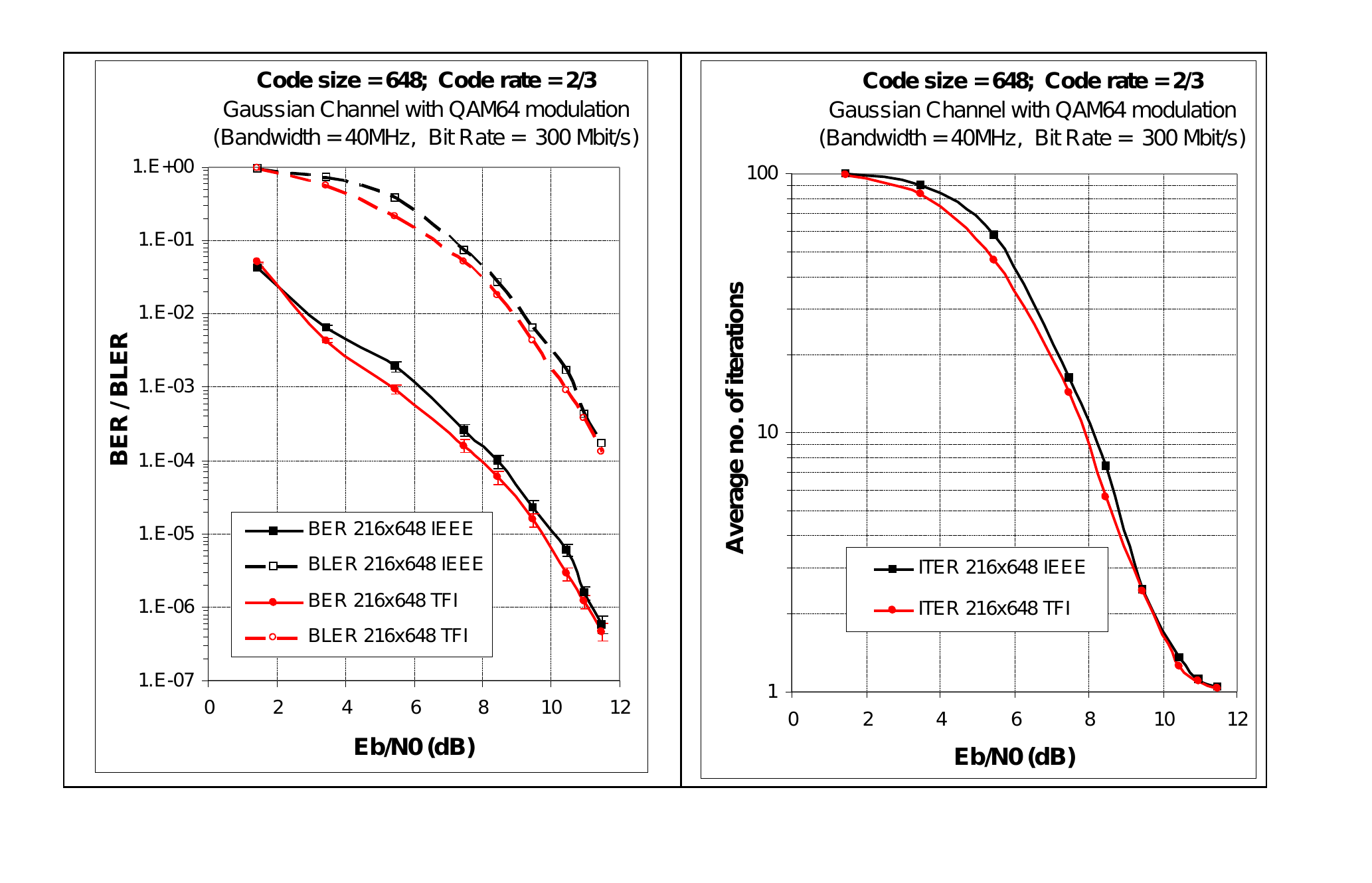}

 Matrix size: 216 by 648; Code size = 648; Code rate = 2/3.
Matrix structure: The last 189 columns contain a `staircase' structure which is
identical as in the IEEE matrix. The remaining part was generated
using the group ring algebraic algorithm which takes 15 initial parameters as
input. 
\subsubsection{Case 2, (802.11n):} 
\includegraphics[scale=.64]{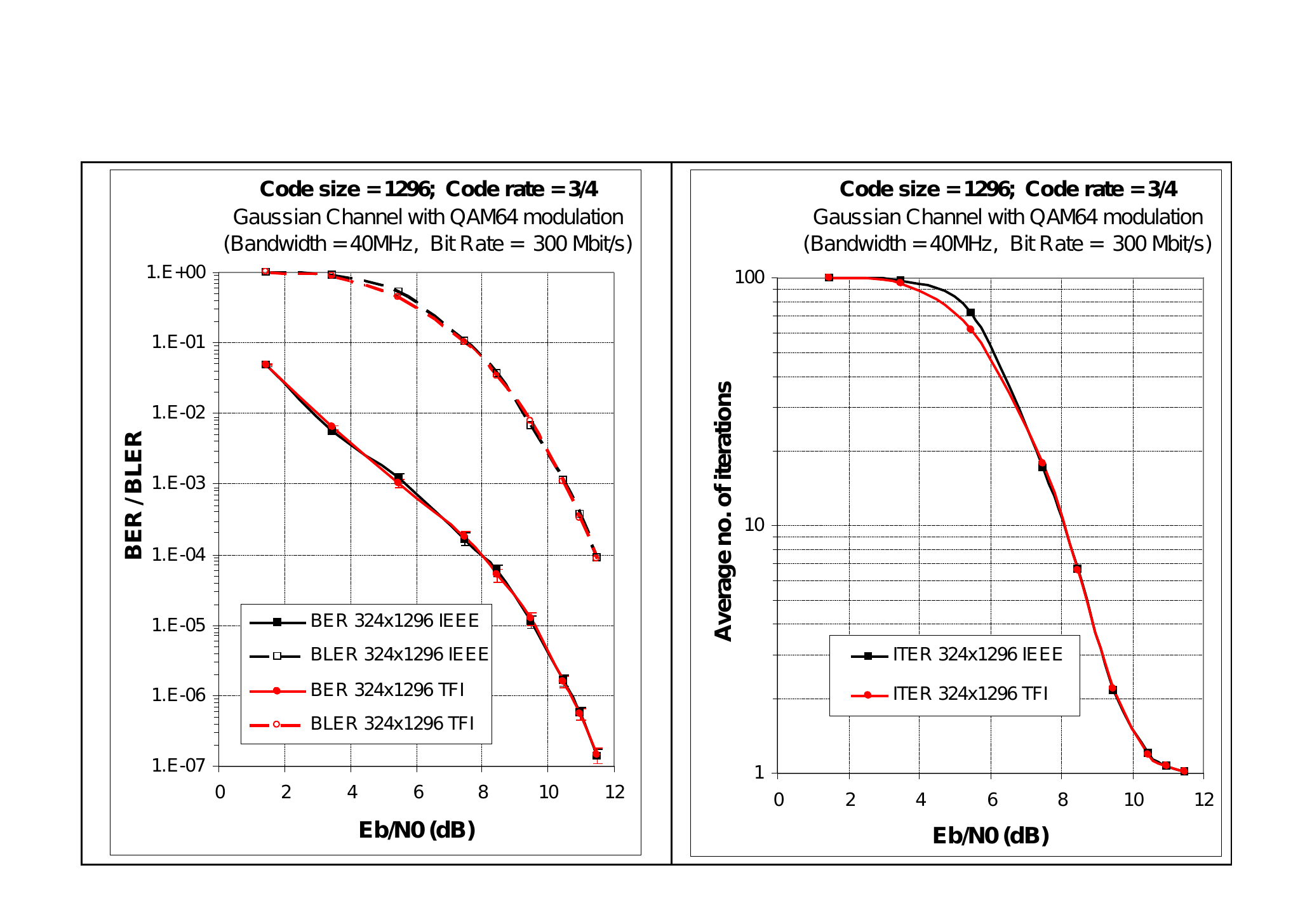}

Matrix size: 324 by 1296; Code size = 1296; Code rate = 3/4.
Matrix structure: The last 270 columns contain a `staircase' structure
which is identical  as in the IEEE matrix. The remaining part
was generated using the  algebraic group ring algorithm which takes 17 initial
parameters as input.

\subsubsection{Case 3, (802.16e):} 

\includegraphics[scale=.7]{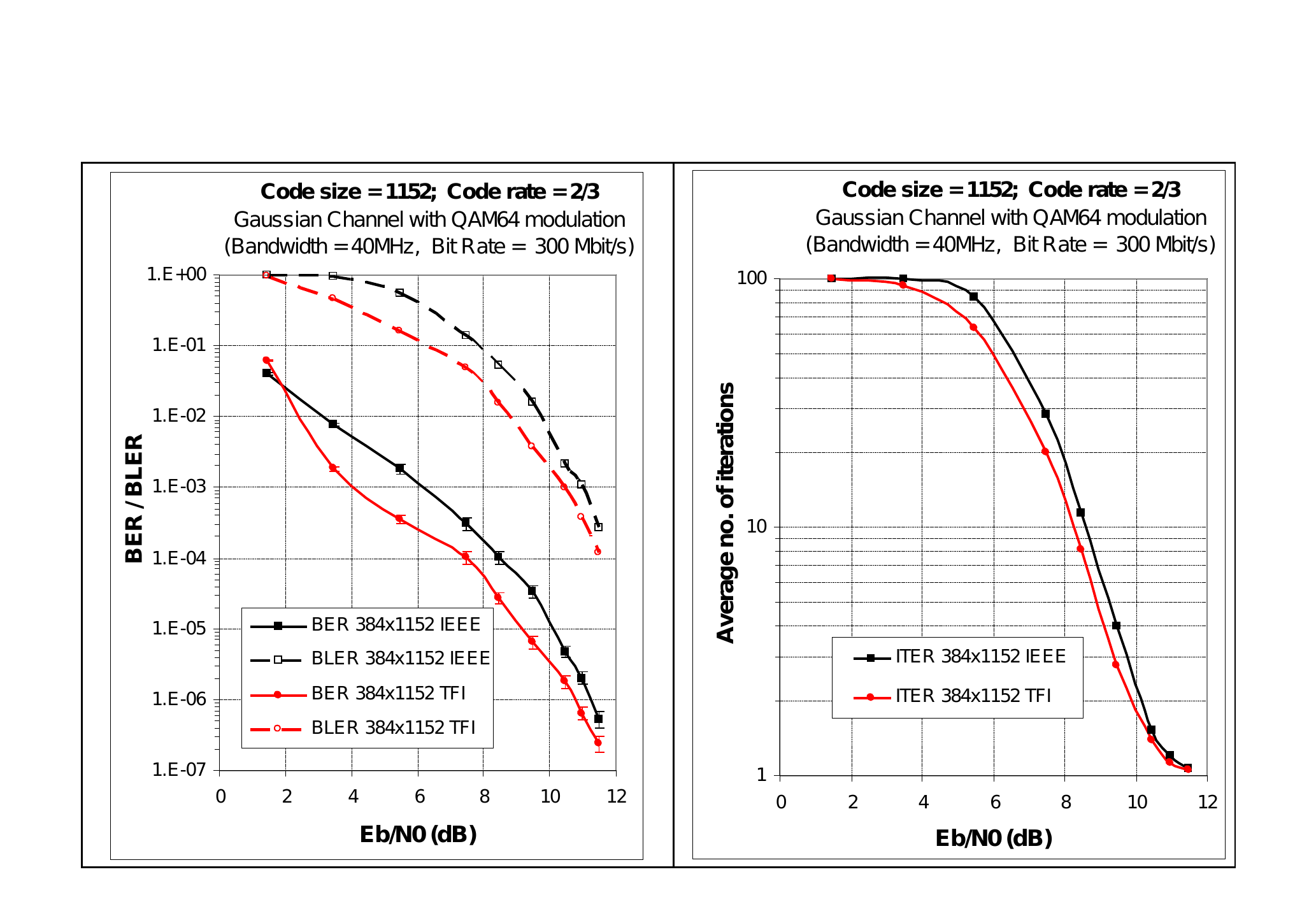} 

Matrix size: 384 by 1152; Code size
= 1152; Code rate = 2/3.
Matrix structure: The last 336 columns contain a `staircase'  
structure which is identical as in the IEEE matrix. The remaining part
was generated using the  algebraic group ring algorithm which takes 17 initial
parameters as input.

\end{document}